\documentclass{amsart}

\usepackage{bm}

 \usepackage{upgreek}

\usepackage{amsopn, amsthm, amsgen, amscd,amsmath, amssymb}
\DeclareMathAlphabet{\mathbbm}{U}{bbm}{m}{n}
\usepackage{color}
\usepackage{tikz}
  \usepackage{graphicx}
\usepackage{epstopdf}
\definecolor{CadetBlue}{cmyk}{0.62, 0.57, 0.23, 0 }
\definecolor{black}{cmyk}{1, 0.5, 0, 0 }
\definecolor{RedViolet}{cmyk}{0.07, 0.9, 0, 0.34 }
\definecolor{SeaGreen}{cmyk}{0.69, 0, 0.5, 0}

\DeclareMathAlphabet{\mathpzc}{OT1}{pzc}{m}{it}

\newcommand{\F}{\mathbb F}

\newcommand{\N}{\mathbb N}

\newcommand{\Q}{\mathbb Q}
\newcommand{\Z}{\mathbb Z}

\newcommand{\e}{\upvarepsilon}

\newtheorem{theo}{Theorem}
\newtheorem{lemm}{Lemma}

\theoremstyle{definition}

\theoremstyle{remark}

\newtheorem{note}{Note}

\title[Quantum $j$-Invariant in Positive Characteristic II]{Quantum $j$-Invariant in Positive Characteristic II: Formulas and Values at the Quadratics}
\author{L. Demangos}
\address{Instituto de Matem\'{a}ticas -- Unidad Cuernavaca, Universidad
Nacional Autonoma de M\'{e}xico, Av. Universidad S/N, C.P. 62210
Cuernavaca, Morelos, M\'{e}xico}
\curraddr{Department of Mathematical Sciences (Mathematics Division),
University of Stellenbosch,
Private Bag X1,
Matieland 7602,
South Africa}
\email{l.demangos@gmail.com}

\author{T.M. Gendron}
\address{Instituto de Matem\'{a}ticas -- Unidad Cuernavaca, Universidad
Nacional Autonoma de M\'{e}xico, Av. Universidad S/N, C.P. 62210
Cuernavaca, Morelos, M\'{e}xico}
\email{tim@matcuer.unam.mx}

\begin{document}
\vspace{2cm}

 \maketitle
 
 \begin{abstract}  
In this sequel to \cite{DGI}, the multi-valued quantum $j$-invariant in positive characteristic is studied at quadratic elements. For every quadratic $f$, 
an explicit expression for each of the values of $j^{\rm qt}(f)$ is given as a limit of rational functions of $f$.  It is proved that the number of values of $j^{\rm qt}(f)$
is finite.
  \end{abstract}

\section{Introduction}

Let $\F_{q}$ be the field with $q$ elements, $q=p^{r}$ a prime power, $k=\F_{q}(T)$
the function field over $\F_{q}$, $k_{\infty}$ the completion 
of $k$ with respect to the valuation $v(f)=-\deg_{T}(f)$. 
In \cite{DGI}, the quantum $j$-invariant was introduced as a multi-valued modular function
\[ j^{\rm qt}: {\rm GL}_{2}(A)\backslash k_{\infty} \multimap k_{\infty} \cup \{ \infty\}\]
obtained as the limit of a sequence of approximating functions 
\begin{align}\label{approximants} j_{\e}: k_{\infty} \longrightarrow k_{\infty} \cup \{ \infty\} \end{align}
as $\e\rightarrow 0$.  Explicit formulas for the approximants (\ref{approximants}) in terms of the sequence $\{ {\tt q}_{i}\}\subset A$ of best approximations of $f$ were 
obtained and using them,
it was shown that $f\in k$ if and only if eventually $ j_{\e}(f)=\infty$. 

In this paper we consider $j^{\rm qt}(f)$ in the special case of $f\in k_{\infty}$ quadratic.  Here we are able to derive explicit formulas
not just for the approximants but for all the values of $j^{\rm qt}(f)$, each of which expressed 
as a limit of certain rational functions of $f$.
Using these formulas we are able to prove that for $f$ quadratic, $\# j^{\rm qt}(f)<\infty$.

In what follows, we fix notation used in \cite{DGI}.

\section{Values at Quadratic Units}

Fix $a\in A-\F_{q}$, $b\in\F_{q}^{\times}$ with 
\[ d:=\deg a>0\] and write $D=a^{2} +4b$.   
The solutions $f,f'$ of $X^{2}-aX-b=0$ are quadratic units, and every quadratic unit in $k_{\infty}-k$
occurs in this way.  
Since $ff'=- b$, we have $|ff'|=1$. Moreover, since  $f+f'=a$, we have as well that $|f|, |f'|\not=1$.  Therefore, if we assume that $|f|>1$, then 
\[ |f|=|a|=q^{d}>1 \quad \text{and}\quad |f'|=|b/a|=q^{-d}<1. \]

Consider the recursive sequence
\[ {\tt Q}_{0}=1, {\tt Q}_{1}=a,\;\dots, \;{\tt Q}_{n+1}=a{\tt Q}_{n}+b{\tt Q}_{n-1} .\]
Note that when $b=1$, ${\tt Q}_{n}={\tt q}_{n}$ is the sequence of best approximations for $f=[a,a,\dots ]$; see \S 3 of \cite{DGI}.
It is clear that for all $n$
\[  
|{\tt Q}_{n}|=|a|^{n}= q^{dn}. \]
Recall 
{\it Binet's formula} 
\[{\tt Q}_{n} =\frac{ f^{n+1}-(f')^{n+1}}{\sqrt{D}},\quad n=0,1,\dots .\]
which, while usually stated over $\Q$, is equally true in this setting, by a very simple induction on $n$.  The root $\sqrt{D}$ is chosen
so that in odd characteristic, $f=(a+\sqrt{D})/2$; in even characteristic, we simply take $\sqrt{D}=a$.

Recall that $\|x\|=$ distance of $x$ to the nearest element of $A$; see \cite{DGI}, \S 2 for this and other relevant notation.
Binet's formula immediately gives 
\begin{align}\label{Q_nerrorestimate} \| {\tt Q}_{n}f\| =  q^{ -(n+1) d } : \end{align}
indeed, 
\begin{align}\label{Q_nerrorestimate2}  |{\tt Q}_{n}f - {\tt Q}_{n+1}| = 
 \frac{|f' |^{n+1} |f-f'|}{|\sqrt{D}|}  =
 |f' |^{n+1}
=q^{ -(n+1) d }<1 .
\end{align}
In particular,  $ {\tt Q}_{n}^{\perp} = {\tt Q}_{n+1}$, where $\uplambda^{\perp}$ denotes the element of $A$ closest to $\uplambda f$.

\vspace{3mm}

The remainder of this section is devoted to deriving an explicit formula for each value of $j^{\rm qt}(f)$.
In view of our need to work with monic polynomials, we make the following modification to the above sequence.
For each $n$, we write $\bar{\tt Q}_{n}$ for the unique monic polynomial obtained as $c_{n}{\tt Q}_{n}$ for some $c_{n}\in\F_{q}^{\times}$.
Then $|\bar{\tt Q}_{n}|=|{\tt Q}_{n}|$ and $\| \bar{\tt Q}_{n}f\|=\| {\tt Q}_{n}f\| $ since $\bar{\tt Q}_{n}^{\perp}=c_{n} {\tt Q}_{n}^{\perp}$.
(N.B. $\bar{\tt Q}_{n}^{\perp}$ is not necessarily monic.)
Note that if we choose $c\in\F_{q}^{\times}$ so that $ca$ is monic then $c_{n}=c^{n}$.
 It follows from Binet's formula that 
\begin{align}\label{barbinet}\bar{\tt Q}_{n} =c^{n}\frac{ f^{n+1}-(f')^{n+1}}{\sqrt{D}}= \frac{ \bar{f}^{n+1}-(\bar{f}')^{n+1}}{c\sqrt{D}},\quad
\bar{f}:=cf,\;\; \bar{f}':=cf'.\end{align}

The set
\begin{align}\label{Tpowbasis} 
\mathcal{B}=\{ T^{d-1}\bar{\tt Q}_{0}, \dots , T\bar{\tt Q}_{0}, \bar{\tt Q}_{0} ; T^{d-1}\bar{\tt Q}_{1}, \dots , T\bar{\tt Q}_{1}, \bar{\tt Q}_{1};\dots  \}\end{align}
is a basis for $A$ as an $\F_{q}$-vector space: as the degree map $\deg:\mathcal{B}\rightarrow\N=\{ 0,1,\dots \}$ is a bijection.
The order in which we have presented the elements of $\mathcal{B}$
 corresponds to decreasing errors:
for $0\leq l\leq d-1$ 
and $n\geq 0$ 
\begin{align}\label{basiserrors}\| T^{l}\bar{\tt Q}_{n}f \| =|T^{l}\bar{\tt Q}_{n}f-T^{l}\bar{\tt Q}_{n+1}|= q^{l-(n+1)d}<1.  \end{align}
In particular, $(T^{l}\bar{\tt Q}_{n})^{\perp}= T^{l} \bar{\tt Q}_{n}^{\perp}$, and the map \begin{align}\label{errorbijection}  \mathcal{B}\longrightarrow q^{-\N},\quad  T^{l}\bar{\tt Q}_{n} \longmapsto
\| T^{l}\bar{\tt Q}_{n}f \| \end{align}
defines a bijection between $\mathcal{B}$ and the set of possible errors. 

Write 
 \[ \mathcal{B}(i) = \{ T^{d-1}\bar{\tt Q}_{i},\dots , \bar{\tt Q}_{i}\}\]  
 for the $i$th block of $\mathcal{B}$.
Furthermore, for $0\leq \tilde{d}\leq d-1$, denote
\[ \mathcal{B}(i)_{\tilde{d}} = \{T^{\tilde{d}}\bar{\tt Q}_{i},\dots , \bar{\tt Q}_{i}\}.\]

\begin{lemm}\label{explicitLambda}  Let $l\in \{ 0,\dots ,d-1\}$
and write
\[ d_{l}=d-1-l.\]   Then 
\begin{align*} 
 \Uplambda_{q^{-Nd-l}}(f )={\rm span}_{\F_{q}} ( \mathcal{B}(N)_{d_{l}},\mathcal{B}(N+1),\dots   ) .
\end{align*}
\end{lemm}

\begin{proof}  First observe that 
\begin{align*}{\rm span}_{\F_{q}} ( \mathcal{B}(N)_{d_{l}},\mathcal{B}(N+1),\dots   ) \subset  \Uplambda_{q^{-Nd-l}}(f ).\end{align*}
Indeed, for $T^{\tilde{d}}\bar{\tt Q}_{N+r}$ with $\tilde{d}\leq d-1$, and $\tilde{d}\leq d_{l}$ if $r=0$, we have 
\[ \| T^{\tilde{d}}\bar{\tt Q}_{N+r}f\| \leq q^{d_{l}}\cdot q^{-(N+1)d} =q^{-Nd-l-1} < q^{-Nd-l}.\]
Moreover, by (\ref{basiserrors}), $ \Uplambda_{q^{-Nd-l}}(f )$ contains no other elements of $\mathcal{B}$. 
In view of the bijection (\ref{errorbijection}), no linear combination of the excluded basis elements could appear
in $ \Uplambda_{q^{-Nd-l}}(f )$.  This proves the equality in the statement of the Lemma.
\end{proof}

We recall the main definitions established in \cite{DGI}.  For each $\upvarepsilon$, 
 $\Uplambda^{\rm mon}_{\upvarepsilon}(f)$ is the subset of monic polynomials in $\Uplambda_{\upvarepsilon}(f) $.
We defined  
\[ \upzeta_{f,\upvarepsilon}(n):=\sum_{\uplambda\in \Uplambda^{\rm mon}_{\upvarepsilon}(f)-\{0\}}\uplambda^{-n}, \quad n\in \N , \]
and 
\[  j_{\upvarepsilon}(f) = \frac{1}{\frac{1}{T^{q}-T}  -J_{\upvarepsilon}(f) }, \]
where 
\[  J_{\upvarepsilon}(f)= \frac{T^{q^{2}}-T}{(T^{q}-T)^{q+1}} \cdot \tilde{J}_{\upvarepsilon}(f) ,\quad 
\tilde{J}_{\upvarepsilon}(f):=\frac{\upzeta_{f,\e}(q^{2}-1)}{\upzeta_{f,\e}(q-1)^{q+1}}.
\]
Then $j^{\rm qt}(f)$ is the set of limits of $ j_{\upvarepsilon}(f)$ for $\upvarepsilon\rightarrow 0$.
It follows from Lemma \ref{explicitLambda} that to calculate the values of $j^{\rm qt}(f )$, 
it suffices to consider the possible limits of $j_{\upvarepsilon}(f)$ formed from the spans of the initial truncations
of $\mathcal{B}$.   

Let $\upvarepsilon=q^{-Nd-l}$.
As in \cite{DGI}, we decompose
\[  \Uplambda^{\rm mon}_{\upvarepsilon} (f)=  \Uplambda^{\rm bas}_{\upvarepsilon}(f)\bigsqcup  \Uplambda^{\rm non-bas}_{\upvarepsilon}(f)  \]
where
\[  \Uplambda^{\rm bas}_{\upvarepsilon}(f)=\{ \mathcal{B}(N)_{d_{l}},\mathcal{B}(N+1), \dots \}  \]
and \[  \Uplambda^{\rm non-bas}_{\upvarepsilon}(f)=  \Uplambda^{\rm mon}_{\upvarepsilon} (f)- \left(\Uplambda^{\rm bas}_{\upvarepsilon} (f)\cup \{0\}\right).\]
Any element $\uplambda\in  \Uplambda^{\rm non-bas}_{\upvarepsilon}(f)$ may therefore be written in the form
\[\uplambda =  c_{0}(T)\bar{\tt Q}_{N} +c_{1}(T)\bar{\tt Q}_{N+1}+ \cdots + c_{m}(T) \bar{\tt Q}_{N+m},\]
where the $c_{i}(T)\in A$ are polynomials that are not all zero and satisfy the following conditions
\begin{itemize}
\item[I.] $c_{m}(T) $ is monic and if $\uplambda =c_{m}(T) \bar{\tt Q}_{N+m}$, $c_{m}(T)\not=T^{j}$ for all $j\leq d-1$. 
 \item[II.] $\deg (c_{i}(T))\leq d-1$, $i=0,\dots , m$.
 \item[III.]  $\deg (c_{0}(T))\leq d_{l}$.
\end{itemize}
In view of this characterization, sums over $\uplambda\in \Uplambda^{\rm non-bas}_{\upvarepsilon}(f )$ may be understood as sums indexed by tuples
 $c_{0}(T),\dots ,c_{m}(T)$ subject to conditions I, II and III, and so we will abbreviate
\[ \sum_{\text{Conditions} \atop \text{I,II,III}}:=\sum_{\Uplambda^{\rm non-bas}_{\upvarepsilon}(f )}.\] 
Using this notation, we may now express 
\begin{align}\label{Jepsilonffirstform} \tilde{J}_{\upvarepsilon}(f)=\frac{ \sum_{\uplambda\in \Uplambda^{\rm bas}_{\upvarepsilon} (f ) }\uplambda^{1-q^{2}}  +
\sum_{\text{Conditions} \atop \text{I,II,III}} 
\big( \sum_{i=0}^{m} c_{i}(T)\bar{\tt Q}_{N+i}\big)^{1-q^{2}}   }{ \left(\sum_{\uplambda\in \Uplambda^{\rm bas}_{\upvarepsilon} (f ) }\uplambda^{1-q}
+
\sum_{\text{Conditions} \atop \text{I,II,III}}
\big( \sum_{i=0}^{m} c_{i}(T)\bar{\tt Q}_{N+i}\big)^{1-q}
 \right)^{q+1} }.\end{align}

Write
\[ \upzeta_{T,l}(n) := 1+ T^{-n} + \cdots + T^{-nd_{l}}, \quad
 \upzeta_{T}(n) :=  \upzeta_{T,0}(n),\]
then define 
\begin{align*}H_{l}(n) := \upzeta_{T,l}(n) +\sum_{\text{Conditions I,II,III}\atop c_{0}(T)\not=0} \left( \sum_{i= 0}^{m} c_{i}(T)\bar{f}^{i}\right)^{-n}  \end{align*}
and
\begin{align*} H(n) :=\frac{ \upzeta_{T}(n)}{\bar{f}^{n}-1} +\sum_{ \text{Conditions}\atop \text{I,II}} \left( \sum_{i=1}^{m} c_{i}(T)\bar{f}^{i}\right)^{-n}.\end{align*}
Now let  \begin{align*}  \tilde{J}(f )_{l}:= \frac{ 
H_{l}(q^{2}-1)+ H(q^{2}-1)
  }{ 
 \left(  
H_{l}(q-1)+  H(q-1)\right)^{1+q}}.\end{align*}
Note that $ \tilde{J}(f )_{l}$  converges: indeed both numerator and denominator have absolute value $1$, and  $| \tilde{J}(f )_{l}|=1$.
Finally, we write 
\[ J(f)_{l}:=\frac{T^{q^{2}}-T}{(T^{q}-T)^{q+1}}\cdot \tilde{J}(f)_{l} .\]
\begin{theo}\label{finitenesstheospecquad}  Let $f $ be a quadratic unit which is a solution of $X^{2}-aX-b$ with $\deg a=d>0$ and $\deg b=0$. Then $j^{\rm qt}(f)$ has precisely
$d=\deg a=\log_{q} |\sqrt{D}|$ values, and they are
\[  j^{\rm qt}(f )=
 \left\{ 
 j(f)_{l} := \frac{1}{\frac{1}{T^{q}-T}  -J(f)_{l} }
\right\}_{l=0}^{d-1} .\]
\end{theo}

\begin{proof}  The conjugate solution satisfies $f'=-bf^{-1}$ with $b\in\F_{q}^{\times}$, so $f'$ is ${\rm GL}_{2}(A)$-equivalent to $f$, hence by modularity,
$j^{\rm qt}(f ')= j^{\rm qt}(f )$. Thus, in what follows, we will suppose $f$ satisfies $|f|>1$.  As above,  $\upvarepsilon = q^{-Nd-l}$.  We begin by noting that
 \[    \sum_{\uplambda\in \Uplambda^{\rm bas}_{\upvarepsilon} (f ) }\uplambda^{-n} = 
  \upzeta_{T,l}(n) \bar{\tt Q}_{N}^{-n} +  
  \upzeta_{T}(n) \sum^{\infty}_{i= 1}\bar{\tt Q}^{-n}_{N+i}.
 \]
Then if we define 
 \[ H_{N,l}(n):=   \upzeta_{T,l}(n) \bar{\tt Q}_{N}^{-n}+\sum_{\text{Conditions I,II,III}\atop c_{0}(T)\not=0} \left( \sum_{i= 0}^{m} c_{i}(T)\bar{\tt Q}_{N+i}\right)^{-n}\]
 and 
 \[  H_{N}(n) := \upzeta_{T}(n) \sum^{\infty}_{i= 1}\bar{\tt Q}^{-n}_{N+i} +\sum_{\text{Conditions}\atop \text{I,II}} \left( \sum_{i=1}^{m} c_{i}(T)\bar{\tt Q}_{N+i}\right)^{-n} ,\]
we may re-write (\ref{Jepsilonffirstform}) as
 \begin{align}\label{Jepsilonffirstform2}\tilde{J}_{\upvarepsilon}(f) = \frac{ H_{N,l}(q^{2}-1) +  H_{N}(q^{2}-1) }{\left( H_{N,l}(q-1) +  H_{N}(q-1)\right)^{q+1}}.\end{align}
 Replace in (\ref{Jepsilonffirstform2}) each $\bar{\tt Q}_{N+i}$ by $(\bar{f}^{N+i+1}-\bar{f}'^{N+i+1})/c\sqrt{D}$ using (\ref{barbinet}): equivalently (since the numerator and denominator of (\ref{Jepsilonffirstform2})
have homogeneous degree $1-q^{2}$) we may omit the constant $1/c\sqrt{D}$ and simply replace $\bar{\tt Q}_{N+i}$ by $\bar{f}^{N+i+1}-\bar{f}'^{N+i+1}$.
  Then dividing out the numerator and denominator by $\bar{f}^{(1-q^{2})(N+1)}$
yields 
   \begin{align}\label{Jepsilonffirstform3}\tilde{J}_{\upvarepsilon}(f) = \frac{ \hat{H}_{N,l}(q^{2}-1) +  \hat{H}_{N}(q^{2}-1) }{\left( \hat{H}_{N,l}(q-1) +  \hat{H}_{N}(q-1)\right)^{q+1}},\end{align}
  where
  \begin{align*} \hat{H}_{N,l}(n):=  & \upzeta_{T,l}(n) \left(1-(\bar{f}'/\bar{f})^{N+1}\right)^{-n} \\
  &+ \sum_{\text{Conditions I,II,III}\atop c_{0}(T)\not=0} \left( \sum_{i= 0}^{m} c_{i}(T)
   \bar{f}^{i}\left(1-(\bar{f}'/\bar{f})^{N+i+1}\right)
  \right)^{-n}\end{align*}
 and 
  \begin{align*}  \hat{H}_{N}(n) := &  \bigg\{ \upzeta_{T}(n) \sum^{\infty}_{i= 1}\bar{f}^{-ni}\left(1-(\bar{f}'/\bar{f})^{N+i+1}\right)^{-n} \\
  & +\sum_{\text{Conditions}\atop \text{I,II}} \left( \sum_{i=1}^{m} c_{i}(T) \bar{f}^{i}\left(1-(\bar{f}'/\bar{f})^{N+i+1}\right)\right)^{-n}.  \end{align*}
  Since $|\bar{f}'/\bar{f}|<1$,
  \[ 1-(\bar{f}'/\bar{f})^{N+i+1}\longrightarrow 1\] 
  as $N\rightarrow\infty$, uniformly in $i$,  and it follows that 
  \[ \lim_{N\rightarrow\infty} \hat{H}_{N,l}(n) =H_{l}(n), \quad  \lim_{N\rightarrow\infty} \hat{H}_{N}(n) =H(n)  \]
  and therefore
 \[  \lim_{N\rightarrow\infty} \tilde{J}_{q^{-Nd-l}}(f)=\tilde{J}(f)_{l} .\]  
\end{proof}

\begin{note}  In the number field case, PARI GP experiments  \cite{Ge-C} performed on fundamental quadratic units $\uptheta$ indicate that
 $j^{\rm qt}(\uptheta )$ appears to have $D$ values, where $D$ is the corresponding fundamental discriminant. 
\end{note}

\section{Arbitrary Real Quadratics}

  Let $h\in k_{\infty}-k$ be an arbitrary quadratic.  By the Dirichlet Unit Theorem \cite{Co}, the group of units in the quadratic extension $k(h)$ is isomorphic to $\F_{q}^{\times}\times \Z$.  Thus there exists
a unit $f$ such that $k(h)=k(f)$, i.e., $h$ can be written in the form
\begin{align}\label{hintermsoff} h=\frac{x+yf}{z},\quad x,y,z\in A, \end{align}
where $f$ satisfies $X^{2}-aX-b=0$, $d=\deg (a)>0$ and $\deg (b)=0$.  

Note that for $h$ and $f$ as in (\ref{hintermsoff}) and  $\upvarepsilon$ satisfying $|z|\upvarepsilon <1$,  we have the following inclusions of $\F_{q}$-vector spaces 
\begin{align}\label{twoinclusions}
z\Uplambda_{|y|^{-1}\upvarepsilon}(f)\subset \Uplambda_{\upvarepsilon}(h)\subset y^{-1}\Uplambda_{|z|\upvarepsilon}(f).
\end{align} 
 The first inclusion follows upon
noting that if $z\uplambda \in z\Uplambda_{|y|^{-1}\upvarepsilon}(f )$ we have 
\[  \| (z\uplambda) h \| =| (z\uplambda)h-(y\uplambda^{\perp} +\uplambda x) | =| \uplambda(x+yf) -(y\uplambda^{\perp} +\uplambda x) | = |y||\uplambda f-\uplambda^{\perp}|<\upvarepsilon  .\]
The second inclusion follows from noting that if $\uplambda\in  \Uplambda_{\upvarepsilon}(h )$, $y\uplambda\in \Uplambda_{|z|\upvarepsilon}(f)$, as
\[  \| (y\uplambda)f\| = |  (y\uplambda)f-(\uplambda^{\perp}z-\uplambda x) | = |\uplambda(zh-x)-(\uplambda^{\perp}z-\uplambda x) |=
|z||\uplambda h-\uplambda^{\perp}| < |z|\upvarepsilon.\]  

\begin{lemm}\label{bndeddim} The inclusion of $\F_{q}$-vector spaces
\[ z\Uplambda_{|y|^{-1}\upvarepsilon}(f)\subset \Uplambda_{\upvarepsilon}(h)\]
has index bounded by a constant which depends only on $y$ and $z$. 
 \end{lemm}

\begin{proof}  By (\ref{twoinclusions}), it suffices to show that the induced inclusion 
\[  z\Uplambda_{|y|^{-1}\upvarepsilon}(f)\subset y^{-1}\Uplambda_{|z|\upvarepsilon}(f)\]
is of index bounded by a constant which only depends on $y,z$. 
For any $\updelta<1$, denote 
\[ \Uplambda^{2}_{\updelta}(f) := \{ (\uplambda ,\uplambda^{\perp})\in A^{2}|\; \uplambda\in \Uplambda_{\updelta}(f)  \}.\]
Note that the map $\uplambda \mapsto (\uplambda,\uplambda^{\perp})$ induces an isomorphism $\Uplambda_{\updelta}(f)\cong \Uplambda^{2}_{\updelta}(f)$ of $\F_{q}$-vector spaces:
this in fact follows from Proposition 1 of \cite{DGI}.   In turn, we obtain an induced isomorphism of $y^{-1}$-rescalings 
\[ y^{-1}\Uplambda_{\updelta}(f)\cong y^{-1}\Uplambda^{2}_{\updelta}(f).\]
Taking $\updelta=|z|\upvarepsilon$, the corresponding isomorphism takes
$ z\Uplambda_{|y|^{-1}\upvarepsilon}(f)$ to $ z\Uplambda^{2}_{|y|^{-1}\upvarepsilon}(f) $.
  Thus
\[ y^{-1}\Uplambda_{|z|\upvarepsilon}(f)/z\Uplambda_{|y|^{-1}\upvarepsilon}(f)\cong  y^{-1}\Uplambda^{2}_{|z|\upvarepsilon}(f)/z\Uplambda^{2}_{|y|^{-1}\upvarepsilon}(f) .\]
On the other hand, the natural map
\[  y^{-1}\Uplambda^{2}_{|z|\upvarepsilon}(f)/z\Uplambda^{2}_{|y|^{-1}\upvarepsilon}(f)\hookrightarrow (y^{-1}A)^{2}/(zA)^{2} \]
is injective: for if $(\uplambda_{1},\uplambda^{\perp}_{1}),(\uplambda_{2},\uplambda^{\perp}_{2})\in \Uplambda^{2}_{|z|\upvarepsilon}(f)$ satisfy 
\[ y^{-1}(\uplambda_{1},\uplambda^{\perp}_{1})-y^{-1}(\uplambda_{2},\uplambda^{\perp}_{2}) =z (\upbeta_{1},\upbeta_{2})\in 
 (zA)^{2},\] then since $z (\upbeta_{1},\upbeta_{2})\in y^{-1} \Uplambda^{2}_{|z|\upvarepsilon}(f)$ (being a difference of elements of the latter),  
 \[    | (yz\upbeta_{1})f-yz\upbeta_{2} |< |z|\upvarepsilon\]
 which implies
 $(\upbeta_{1},\upbeta_{2})\in \Uplambda^{2}_{|y|^{-1}\upvarepsilon}(f )$.
Therefore the index of $z\Uplambda_{|y|^{-1}\upvarepsilon}(f)$ in $y^{-1}\Uplambda_{|z|\upvarepsilon}(f) $
is bounded by $\dim_{\F_{q}}((y^{-1}A)^{2}/(zA)^{2})$.
 This proves the Lemma. 
 \end{proof}

\begin{theo}\label{finitenessofjtheo}  Let $h\in k_{\infty}-k$ be quadratic.  Then $\# j^{\rm qt}(h)<\infty$.
\end{theo}

\begin{proof}  
By Lemma \ref{bndeddim},
for each $\upvarepsilon$, we may find $\uplambda_{1,\upvarepsilon}=0, \dots , \uplambda_{r,\upvarepsilon}\in \Uplambda_{\upvarepsilon}(h )$,
 $r<M=M(y,z)$, such that
\begin{align}\label{cosetdecomp} \Uplambda_{\upvarepsilon}(h ) =   \big[  z\Uplambda_{|y|^{-1}\upvarepsilon}(f)+\uplambda_{1,\upvarepsilon}\big] \bigsqcup \cdots \bigsqcup
\big[ z\Uplambda_{|y|^{-1}\upvarepsilon}(f)+\uplambda_{r,\upvarepsilon}\big] .\end{align}
Thus,  we have
\begin{align*} \tilde{J}_{\upvarepsilon}(h)  = \frac{\upzeta_{h,\e}(q^{2}-1)}{\upzeta_{h,\e}(q-1)^{q+1}} & =- \frac{\sum_{\uplambda\in \Uplambda_{\e}(h)-0} 
\uplambda^{1-q^{2}}}{\left(\sum_{\uplambda\in \Uplambda_{\e}(h)-0} 
\uplambda^{1-q}\right)^{q+1}} \\
&=-
\frac{ \sum_{i=1}^{r}\sum_{\uplambda\in \Uplambda_{|y|^{-1}\upvarepsilon}(f ) } (\uplambda + \uplambda_{i,\upvarepsilon}/z)^{1-q^{2}}  }{\left( \sum_{i=1}^{r}\sum_{\uplambda\in \Uplambda_{|y|^{-1}\upvarepsilon}(f ) } (\uplambda + \uplambda_{i,\upvarepsilon}/z)^{1-q} \right)^{1+q}}.
\end{align*}
By (\ref{twoinclusions}), for $i\not=1$,
\begin{align}\label{formofai}
\uplambda_{i,\upvarepsilon}\in  y^{-1}\Uplambda_{|z|\upvarepsilon}(f )-z\Uplambda_{|y|^{-1}\upvarepsilon}(f),
\end{align}
so we may write
\begin{align}\label{formoftheai}  \uplambda_{i,\upvarepsilon} = y^{-1}\sum_{j=m}^{m'}  c_{ij,\upvarepsilon}(T)\bar{\tt Q}_{i} ,\end{align}
where $\deg (c_{ij}(T))\leq d-1$ and $m$ is the smallest index so that $\bar{\tt Q}_{m}\in \Uplambda_{|z|\upvarepsilon}(f )$. 
Note that by (\ref{formofai}) and Lemma \ref{bndeddim}, the difference
$m'-m$ has a uniform bound which
is independent of $\upvarepsilon$.
In particular, the set of all possible coefficients $c_{ij,\upvarepsilon}(T)$ is finite in number.

Consider a sequence $\{ q^{-Nd-l}\}_{N>1}$ of values for $\upvarepsilon$ giving a limiting value for $j^{\rm qt}(f)$ as described in Theorem \ref{finitenesstheospecquad}: recall that the limit is calculated by using Binet's formula to replace occurrences of $\bar{\tt Q}_{N+i}$ by $\bar{f}^{N+i+1}$ in $\tilde{J}_{\upvarepsilon}(f)$ and dividing
out by $\bar{f}^{(1-q^{2})(N+1)}$.  
We apply exactly the same
process to $\tilde{J}_{\upvarepsilon}(h)$: and in view of (\ref{formoftheai})
we obtain approximations
\begin{align}\label{aiapprox} z^{-1} \frac{\uplambda_{i,\upvarepsilon}}{\bar{f}^{N+1}} \sim (yz)^{-1} \sum_{j=m}^{m'} c_{ij,\upvarepsilon}(T)\bar{f}^{j-N} ,\quad m,m'\geq 0 \end{align}
where $\deg c_{ij,\upvarepsilon}(T)\leq d-1$.  Therefore, the set of possible limits of the $z^{-1} \frac{\uplambda_{i,\upvarepsilon}}{\bar{f}^{N+1}}$ as $N\rightarrow \infty$ is contained
in the finite set
\[  \left\{ \left. (yz)^{-1} \sum_{j=-n}^{n'} c_{j}(T)\bar{f}^{j}\right|\quad \deg(c_{j}(T))\leq d-1\right\} \]
where $n'+n = m'-m$.
Thus, within the family $\{ \upvarepsilon=q^{-Nd-l}\}_{N>1}$, there are only finitely many possible limits of sub-sequences of 
  $\{ \tilde{J}_{\upvarepsilon}(h) \}$
giving rise to elements of $j^{\rm qt}(h )$.
This proves the Theorem.
 \end{proof}

\end{document}